\RequirePackage{fix-cm}
\documentclass[smallextended]{svjour3}
\smartqed

\usepackage{amsmath,amssymb,amsfonts,latexsym}

\newcommand{\bbr}{\mathbb{R}}
\newcommand{\bbp}{\mathbb{P}}
\newcommand{\bbe}{\mathbb{E}}
\newcommand{\Ind}{\mathbf{1}}
\newcommand{\diff}{\mathop{}\!\mathrm{d}}
\newcommand{\cF}{\mathcal{F}}
\newcommand{\pup}{\psi_{\uparrow}}
\newcommand{\pdn}{\psi_{\downarrow}}

\journalname{}

\begin{document}

\title{Smoothness of the survival probability in models with a random environment: the annuity and mixed cases}
\titlerunning{Smoothness of the survival probability}

\author{Platon Promyslov}
\authorrunning{P. Promyslov}

\institute{Platon Promyslov \at
           Faculty of Mechanics and Mathematics, Lomonosov Moscow State University, Moscow, Russia \\
           \email{platon.promyslov@gmail.com}}

\date{\today}

\maketitle

\begin{abstract}
We consider the ruin problem for an insurance company investing its whole
reserve in a risky asset whose parameters depend on a Markov random
environment. Using the Green's function method we prove the $C^2$-smoothness
of the survival probability in the model with annuity payments and, under an
additional condition, in the model with two-sided jumps, which makes a
rigorous derivation of the corresponding equations possible. The approach
weakens the assumptions on the jump distribution: for annuities it suffices
that it be an arbitrary probability measure on the positive half-line, and in
the mixed model that it have no atoms on the negative half-line. The latter
condition cannot be dropped: in the presence of an isolated atom the second
derivative has a jump, whose size is computed explicitly.
%\subclass{60G44 \and 91G05 \and 45J05}
\end{abstract}

\section{Introduction}
\label{s1}

The central theoretical problem in ruin problems with investments in risky
assets and switching market regimes, see e.g.\
\cite{AntipovKabanov,DiMasi,Paulsen}, is to justify the smoothness of the
survival probability, which is needed for a correct derivation and use of the
integro-differential equations (IDEs).

In the recent paper \cite[Theorem~1]{AntipovKabanov} smoothness was proved for
the model with positive drift of the capital and downward jumps, the non-life
insurance case, under the assumption that the jump distribution has a twice
continuously differentiable density $f$ with $f',f''\in L^1$. In
\cite[Remark~1]{AntipovKabanov} it was conjectured that the annuity and the
mixed-jump cases can be treated in the same way with minor changes.

We show below that this conjecture needs a refinement. Reversing the direction
of the jumps and the sign of the drift changes the ruin mechanism: instead of
crossing the barrier by a jump, zero is reached continuously. For annuities we
prove smoothness for an arbitrary probability measure of the jumps on
$(0,\infty)$. In the mixed model, on the contrary, smoothness may fail: if the
jump distribution has an isolated atom on the negative half-line and
$\Phi_i(0+)>0$, which holds, for instance, for a positive drift, then the
second derivative has a jump at the corresponding point,
see Proposition~\ref{prop:counterexample}. Smoothness is restored if for every
regime the jump distribution has no atoms on $(-\infty,0)$ or $\Phi_i(0+)=0$,
see Theorem~\ref{thm:mixed_case}. Thus the method developed here allows one to
drop the smoothness assumptions on the density imposed in
\cite[Theorem~1]{AntipovKabanov}: for annuities atoms, singular components and
an infinite expectation are all allowed, because the integral term is never
differentiated.

\section{The model}
\label{s2}

All random objects are defined on a probability space
$(\Omega,\cF,\mathbf{F},\bbp)$ with a filtration
$\mathbf{F}=(\cF_t)_{t\ge0}$. Let $\theta=(\theta_t)_{t\ge0}$ be a homogeneous
Markov process with values in a finite set $E=\{1,\dots,K\}$ and intensity
matrix $\Lambda=(\lambda_{ij})$, describing the changes of the economic
environment. Here $\lambda_{ij}\ge0$ for $i\ne j$, and the rows of $\Lambda$
sum to zero.

The capital $X_t$ of the insurance company with initial value $u>0$ is
invested in a risky asset $S_t$. Its dynamics is given by the stochastic
differential equation
\begin{equation}
    \diff X_t = X_t\frac{\diff S_t}{S_t}+\diff P_t
    = X_t\bigl(a_{\theta_t}\diff t+\sigma_{\theta_t}\diff W_t\bigr)+\diff P_t,
\end{equation}
where $W_t$ is a standard Wiener process and $P_t$ is a jump process
describing the flow of premiums and claims,
\begin{equation}
    \label{eq:dP}
    \diff P_t = c_{\theta_t}\diff t+\int_{\bbr} z\,\pi(\diff t,\diff z).
\end{equation}
Here $c_i$ is the drift coefficient in regime $i\in E$ and
$\pi(\diff t,\diff z)$ is the Poisson random measure of jumps. Its compensator
is $\Pi(\diff t,\diff z)=\alpha_{\theta_t}F_{\theta_t}(\diff z)\diff t$, where
$\alpha_i>0$ is the intensity of the insurance events in regime $i\in E$ and
$F_i(\diff z)$ is the probability distribution of the jump size $z$ in regime
$i$, defined on $\bbr\setminus\{0\}$. The processes $W$, $\pi$ and $\theta$ are
assumed to be mutually independent; for Poisson measures and compensators see
\cite{Sato}. Since $\alpha_i<\infty$, the integral in \eqref{eq:dP} is a finite
sum and no moment assumptions on $F_i$ are needed.

In what follows we assume, for every regime $i\in E$, the non-degeneracy of the
diffusion, $\sigma_i>0$, and the non-triviality, $\Phi_i\not\equiv0$. In the
scalar case the condition $a_i>\sigma_i^2/2$ is sufficient for the latter; it
is equivalent to the integrability of the scale function at infinity, see
\cite{Kabanov} for the annuity model.

The ruin time is $\tau=\inf\{t>0:\ X_t\le0\}$. The main object of study is the
vector of survival probabilities
$\Phi(u)=(\Phi_1(u),\dots,\Phi_K(u))^\top$, where
$\Phi_i(u)=\bbp(\tau=\infty\mid X_0=u,\ \theta_0=i)$. Note first that every
$\Phi_i$ is non-decreasing. Indeed, if $u\le u'$, then for the difference of
the processes built from the same $W$, $\pi$ and $\theta$ we have
$X^{u'}-X^{u}=(u'-u)+\bigl(X^{u'}-X^{u}\bigr)\cdot R$, whence
$X^{u'}_t-X^{u}_t=(u'-u)\mathcal E^i_t>0$, where $\mathcal E^i$ is the
stochastic exponential of $R$. Therefore $\tau^{u}\le\tau^{u'}$. In particular
the $\Phi_i$ are Borel functions and the limits $\Phi_i(0+)\in[0,1]$ exist.
Put $\tilde\lambda_i:=\sum_{j\ne i}\lambda_{ij}+\alpha_i>0$ and
$$
    \mathcal H_i(\Phi)(u) := \sum_{j\ne i}\lambda_{ij}\Phi_j(u)
    + \alpha_i\!\int_{\bbr}\!\Phi_i(u+z)\Ind_{\{u+z>0\}}F_i(\diff z),
$$
so that $0\le\mathcal H_i(\Phi)\le\tilde\lambda_i$.

Depending on the sign of $c_i$ and the support of $F_i$ one distinguishes
three cases. In non-life insurance $c_i>0$, the support of $F_i$ lies in
$(-\infty,0)$, and ruin occurs by a jump. In life insurance, that is, for
annuity payments, $c_i<0$, the support lies in $(0,\infty)$, and ruin is
possible only through a continuous passage to zero. In the mixed model
$c_i\in\bbr$ is arbitrary and the support lies in $\bbr\setminus\{0\}$, so
that jumps of both signs are possible.

\section{Auxiliary results}
\label{s3}

\begin{lemma}[Asymptotics of the solutions]
\label{lem:asymptotics}
Let $\sigma>0$, $c<0$, $\lambda>0$ and
$\mathcal L^0\psi:=\sigma^2u^2\psi''/2+(au+c)\psi'$. Consider the equation
$\mathcal L^0\psi-\lambda\psi=0$. It has a fundamental system of solutions
$\{\pup,\pdn\}$ with $\pup(0+)=0$ and $\pdn$ bounded at infinity, and as
$u\to0+$
\begin{equation}
    \label{eq:psi_asympt}
    \pup(u)=C_\uparrow\,u^{\beta}\exp\left(-\frac{\kappa}{u}\right)
    \bigl(1+O(u)\bigr),\quad \kappa=\frac{2|c|}{\sigma^2},\quad
    \beta=2-\frac{2a}{\sigma^2},
\end{equation}
where $C_\uparrow>0$, while $\pdn(0+)$ is finite and non-zero. As
$v\to\infty$ one has $\pdn(v)\sim C_\downarrow v^{\rho_-}$, where $\rho_-<0$
is the smaller root of
$$
    \tfrac12\sigma^2\rho^2+\bigl(a-\tfrac12\sigma^2\bigr)\rho-\lambda=0 .
$$
The solutions may be chosen positive; then $\pup$ is strictly increasing,
$\pdn$ is strictly decreasing, they are linearly independent, the Wronskian
$\mathcal W=\pup\pdn'-\pup'\pdn$ is strictly negative, and by Liouville's
formula
\begin{equation}
    \label{eq:liouville}
    \mathcal W(v)=C_{\mathcal W}\,v^{-2a/\sigma^2}\,e^{2c/(\sigma^2v)},
    \qquad C_{\mathcal W}<0 .
\end{equation}
For $c\ge0$ the behaviour of the solutions is different. Zero is inaccessible,
$\pdn(0+)=+\infty$, and the role of $\pup$ is played by the solution bounded at
zero. For $c>0$ such a solution is unique up to a factor and satisfies
$\pup(0+)>0$, while every solution not proportional to it grows like
$u^{\beta}e^{2c/(\sigma^2u)}$. For $c=0$ the equation is of Euler type and
$\pup(u)=u^{\rho_+}$, $\pdn(u)=u^{\rho_-}$, where $\rho_+>0$ is the larger root
of the same equation. Formula \eqref{eq:liouville} and the negativity of the
Wronskian hold for every sign of $c$.
\end{lemma}

\begin{proof}
The point $u=0$ is an irregular singular point of rank $1$. Substituting
$\psi(u)=e^{h(u)}$ we obtain for $h'$ the Riccati equation
$$
\tfrac12\sigma^2u^2\bigl(h''+(h')^2\bigr)+(au+c)h'-\lambda=0 .
$$
Inserting $h'(u)=K/u^2+B/u+O(1)$ gives, at the orders $u^{-2}$ and $u^{-1}$,
$$
\tfrac12\sigma^2K^2+cK=0,\qquad \sigma^2K(B-1)+aK+cB=0 .
$$
The first equation yields $K=0$ or $K=-2c/\sigma^2=\kappa>0$. For $K=\kappa$,
that is, for $c=-\sigma^2K/2$, the second one gives
$\sigma^2B/2+a-\sigma^2=0$, whence $B=\beta$. Integration gives
$h(u)=-\kappa/u+\beta\ln u+O(u)$ and the asymptotics \eqref{eq:psi_asympt}. The
existence of solutions with such behaviour is guaranteed by the theory of
singular points of the second kind: the substitution $t=1/u$ transforms the
equation into the system
$$
    \dot y=\bigl(A_0+A_1t^{-1}+A_2t^{-2}\bigr)y ,
    \qquad
    A_0=\begin{pmatrix} 0 & 1\\ 0 & 2c/\sigma^2\end{pmatrix} ,
$$
the characteristic roots of $A_0$ being $0$ and $2c/\sigma^2$, which are
distinct for $c\ne0$, so that Theorem~4.1 of Chapter~V in \cite{Coddington}
applies. The branch $K=0$ corresponds to a solution $\chi$ with the asymptotic
expansion $\chi(u)=b_0+b_1u+\dots$, whose leading coefficient $b_0$ is non-zero
because in that theorem the leading vector of a formal solution is a
characteristic vector of $A_0$. The pair $\{\chi,\pup\}$ is linearly
independent, since $\pup(0+)=0\ne b_0$; hence in the expansion
$\pdn=A\chi+B\pup$ the coefficient $A$ is non-zero, for otherwise $\pdn$ would
be proportional to $\pup$, contradicting $\mathcal W\ne0$. Therefore
$\pdn(0+)=Ab_0$ is finite and non-zero. The behaviour at infinity is obtained
by the substitution $x=\ln u$, which turns the equation into a system
$y'=\bigl(A+R(x)\bigr)y$ with a constant matrix $A$ whose characteristic roots
are the roots of the polynomial above and with a perturbation $R(x)=O(e^{-x})$
integrable on $[x_0,\infty)$. The roots are distinct and of opposite signs,
since their product equals $-2\lambda/\sigma^2<0$, so that Theorem~8.1 of
Chapter~III in \cite{Coddington} applies.

Let now $\psi>0$ and $\psi'(u_1)=0$. Then
$\sigma^2u_1^2\psi''(u_1)/2=\lambda\psi(u_1)>0$, that is, every critical point
of a positive solution is a strict local minimum, and a positive solution has
no local maxima. Hence $\pup$ is strictly increasing and $\pdn$ strictly
decreasing, which gives $\mathcal W<0$; equality \eqref{eq:liouville} follows
by integrating $\mathcal W'=-2(au+c)\mathcal W/(\sigma^2u^2)$ and does not
depend on the sign of $c$.

The case $c\ge0$ is treated in the same way. For $c>0$ the matrix $A_0$ above
has the roots $0$ and $2c/\sigma^2>0$, so that Theorem~4.1 of Chapter~V in
\cite{Coddington} provides a solution with a finite non-zero limit at zero and
a solution growing like $u^{\beta}e^{2c/(\sigma^2u)}$. From the explicit
representation
\begin{equation}
    \label{eq:Yexpl}
    Y_t=\mathcal E_t\Bigl(u+c\int_0^t\mathcal E_s^{-1}\diff s\Bigr) ,
    \qquad
    \mathcal E_t:=\exp\bigl((a-\sigma^2/2)t+\sigma W_t\bigr) ,
\end{equation}
one sees that $Y_t\ge u\mathcal E_t>0$, that is, zero is inaccessible. Since
$\pup$ increases and $\pdn$ decreases, they are not proportional, so exactly
one of them is bounded at zero, namely $\pup$, and $\pdn(0+)=+\infty$. For
$c=0$ the equation is of Euler type and everything is checked directly.
\qed
\end{proof}

\begin{lemma}[Continuity of the integral term]
\label{lem:convolution}
Let the function $\Phi_i$ be bounded and continuous on $(0,\infty)$ and
$$
I_i(u):=\int_{\bbr}\Phi_i(u+z)\,\Ind_{\{u+z>0\}}\,F_i(\diff z).
$$
Then $I_i$ is continuous at a point $u>0$ provided at least one of the
following conditions holds:
\begin{enumerate}
\item[\rm(i)] the support of $F_i$ is contained in $(0,\infty)$;
\item[\rm(ii)] $F_i(\{-u\})=0$;
\item[\rm(iii)] $\Phi_i(0+)=0$.
\end{enumerate}
\end{lemma}

\begin{proof}
Let $u_n\to u>0$. The integrand is bounded by one and, for every $z\ne-u$,
converges to $\Phi_i(u+z)\Ind_{\{u+z>0\}}$: for $z>-u$ by the continuity of
$\Phi_i$, and for $z<-u$ both sides vanish. In case {\rm(i)} the point
$z=-u<0$ does not belong to the support; in case {\rm(ii)} the set $\{-u\}$ has
$F_i$-measure zero. In case {\rm(iii)} the convergence also holds at $z=-u$,
since $\Phi_i(u_n-u)\Ind_{\{u_n>u\}}\to0=\Phi_i(0+)$. It remains to apply the
dominated convergence theorem.
\qed
\end{proof}

\begin{remark}
\label{rem:nodiff}
The identity
$$
I_i'(u)=\int_{\bbr}\Phi_i'(u+z)\,\Ind_{\{u+z>0\}}\,F_i(\diff z)
$$
is, in general, false: the indicator contributes at the points where $u+z=0$.
Below the integral term is never differentiated, only its continuity is used,
and this is why no moment assumptions on $F_i$ are needed.
\end{remark}

\section{Raising the order}
\label{s4}

We show that passing from models with one-sided jumps to the mixed model
raises the order of the equation. Consider the scalar case with exponential
jumps, starting from the IDE derived below in Section~\ref{s6}:
\begin{equation}
    \label{eq:general_scalar}
    \tfrac12\sigma^2u^2\Phi''(u)+(au+c)\Phi'(u)-\alpha\Phi(u)
    +\mathcal I(\Phi)(u)=0.
\end{equation}

\begin{proposition}
\label{prop:order}
Let the jump distributions have exponential densities. Then:
\begin{enumerate}
\item in the non-life and the annuity cases equation \eqref{eq:general_scalar}
reduces to an ODE of the third order; for the non-life case this equation is
written out in \cite[\S\,5]{AntipovKabanov};
\item in the mixed model equation \eqref{eq:general_scalar} reduces to an ODE
of the fourth order.
\end{enumerate}
\end{proposition}

\begin{proof}
\textbf{1.} In the non-life case ($c>0$, downward jumps, density
$\mu e^{\mu z}$ on $z<0$) the integral term equals
$$
\mathcal I(\Phi)(u)=\alpha\mu e^{-\mu u}\int_0^u\Phi(y)e^{\mu y}\diff y .
$$
Applying the operator $\mathcal D_+:=\diff/\diff u+\mu$ we get
$$
\mathcal D_+\mathcal I(\Phi)(u)
=\alpha\mu\Bigl(-\mu e^{-\mu u}\textstyle\int_0^u\Phi(y)e^{\mu y}\diff y
+\Phi(u)\Bigr)+\mu\mathcal I(\Phi)(u)=\alpha\mu\Phi(u).
$$
The integral disappears while the order of the differential part is raised by
one, to the third. For annuities, that is, for upward jumps, the operator
$\mathcal D_-:=\diff/\diff u-\mu$ works in the same way.

\textbf{2.} In the mixed case, a sum of two integrals with parameters $\mu_1$
and $\mu_2$, it suffices to apply the composition
$\mathcal D_+(\mu_1)\mathcal D_-(\mu_2)$: the operators $\mathcal D_\pm$ have
constant coefficients and therefore commute, and the composition annihilates
both integrals. This is an operator of the second order, and applying it to the
leading term $\Phi''$ produces $\Phi^{(4)}$.
\qed
\end{proof}

Raising the order is a property of the reduction to an ODE: the method
presented below does not use it, whereas Proposition~\ref{prop:order} is
essential when asymptotics are analysed via the ODE.

\section{The case of annuity payments}
\label{s5}

Let $c_i<0$ and let the jumps be directed upwards only. Since the jumps are
positive, the process cannot cross the ruin barrier by a jump: the ruin time
$\tau$ coincides with the first passage time to zero and $X_\tau=0$.

We apply the Green's function method: between the events the process behaves
like a diffusion. Let $Y=Y^{u,i}$ be the solution of
$\diff Y=(a_iY+c_i)\diff t+\sigma_iY\diff W$ with $Y_0=u>0$, describing the
evolution of the capital in the absence of events, let
$\zeta:=\inf\{t>0:\ Y_t\le0\}$, and let $T$ be the time of the first event,
that is, of a jump of $P$ or of a switch of the environment. By the
independence of $W$, $\pi$ and $\theta$ the time $T$ does not depend on $W$ and
is exponentially distributed with parameter $\tilde\lambda_i$, the event being a
switch to $j$ with probability $\lambda_{ij}/\tilde\lambda_i$ and a jump with
probability $\alpha_i/\tilde\lambda_i$.

By Lemma~\ref{lem:asymptotics} the homogeneous equation
$\mathcal L^0_i\psi-\tilde\lambda_i\psi=0$ has a solution $\psi_{i,\uparrow}$
vanishing at zero and a solution $\psi_{i,\downarrow}$ bounded at infinity. We
build from them the Green's function
\begin{equation}
    \label{eq:green}
    G_i(u,v)=\frac{2}{\sigma_i^2v^2|\mathcal W_i(v)|}\times
    \begin{cases}
    \psi_{i,\uparrow}(u)\psi_{i,\downarrow}(v), & 0<u\le v, \\
    \psi_{i,\uparrow}(v)\psi_{i,\downarrow}(u), & u>v.
    \end{cases}
\end{equation}
By Lemma~\ref{lem:asymptotics} the Wronskian $\mathcal W_i$ is strictly
negative, in particular it does not vanish, and the kernel \eqref{eq:green} is
positive. From \eqref{eq:psi_asympt} and \eqref{eq:liouville} one sees that in
the ratio $\psi_{i,\uparrow}(v)/\bigl(v^2|\mathcal W_i(v)|\bigr)$ the
exponentials cancel, since $e^{2c_i/(\sigma_i^2v)}=e^{-\kappa_i/v}$, and so do
the powers, since $\beta_i+2a_i/\sigma_i^2=2$. Hence the limit as $v\to0$ is
finite. As $v\to\infty$ we have
$G_i(u,v)=O\bigl(v^{\rho_--2+2a_i/\sigma_i^2}\bigr)$, and integrability is
equivalent to $\rho_-<\rho^*$, where $\rho^*:=1-2a_i/\sigma_i^2$. This holds
automatically: substituting $\rho^*$ into the characteristic equation gives the
value $-\tilde\lambda_i<0$, and the parabola is negative exactly between its
roots. Thus the Green's function is integrable in $v$, uniformly on compact
subsets of $(0,\infty)$. Put
$$
    \mathcal K_i[f](u):=\int_0^\infty G_i(u,v)\,f(v)\,\diff v .
$$

\begin{lemma}[Exact integral representation]
\label{lem:repr}
For all $i\in E$ and $u>0$ the identity
$\Phi_i=\mathcal K_i\bigl[\mathcal H_i(\Phi)\bigr]$ holds, no smoothness of
$\Phi$ being assumed.
\end{lemma}

\begin{proof}
If $\zeta\le T$, ruin occurs before the first event. If $T<\zeta$, then by the
strong Markov property of the process $(X,\theta)$ at time $T$ the conditional
survival probability equals $\Phi_j(Y_T)$ in the case of a switch to $j$ and
$\Phi_i(Y_T+z)\Ind_{\{Y_T+z>0\}}$ in the case of a jump of size $z$. Averaging
over the type of the event and over $z$, and using that $X$ coincides with $Y$
on $[0,T)$, we obtain
$$
\Phi_i(u)=\frac{1}{\tilde\lambda_i}\,
\bbe\bigl[\Ind_{\{T<\zeta\}}\,\mathcal H_i(\Phi)(Y_T)\bigr] .
$$
Since $T$ does not depend on $(Y,\zeta)$ and all the quantities are bounded,
Fubini's theorem gives
$$
\Phi_i(u)=\bbe\Bigl[\int_0^{\zeta}e^{-\tilde\lambda_i t}\,
\mathcal H_i(\Phi)(Y_t)\,\diff t\Bigr],
$$
and this is exactly $\mathcal K_i[\mathcal H_i(\Phi)](u)$: the right-hand side
is the classical representation of the Green's function of a one-dimensional
diffusion killed at zero, see \cite[Part~I, Ch.~II, \S\S\,1, 4]{Borodin}.
\qed
\end{proof}

\begin{lemma}
\label{lem:LtoC}
The operator $\mathcal K_i$ maps $L^\infty(0,\infty)$ into $C^1(0,\infty)$. If,
in addition, the function $f$ is continuous, then
$\mathcal K_i[f]\in C^2(0,\infty)$ and
\begin{equation}
    \label{eq:resolvent_ode}
    \mathcal L_i^0\,\mathcal K_i[f]-\tilde\lambda_i\,\mathcal K_i[f]=-f
    \qquad \text{on } (0,\infty).
\end{equation}
\end{lemma}

\begin{proof}
Write $y=\mathcal K_i[f]=\psi_{i,\downarrow}A+\psi_{i,\uparrow}B$, where
\begin{gather*}
A(u)=\int_0^u\psi_{i,\uparrow}f\,m_i\diff v ,
\qquad
B(u)=\int_u^\infty\psi_{i,\downarrow}f\,m_i\diff v , \\
m_i(v):=2\bigl/\bigl(\sigma_i^2v^2|\mathcal W_i(v)|\bigr) .
\end{gather*}
Both integrals are finite by the above. The functions $A$ and $B$ are locally
Lipschitz, with $A'=\psi_{i,\uparrow}f\,m_i$ and
$B'=-\psi_{i,\downarrow}f\,m_i$ almost everywhere, so that the terms
$\pm\psi_{i,\downarrow}\psi_{i,\uparrow}f\,m_i$ arising in the differentiation
cancel identically and
$y'=\psi_{i,\downarrow}'A+\psi_{i,\uparrow}'B$ almost everywhere. The function
$y$ is absolutely continuous and the right-hand side is continuous, hence the
identity holds everywhere and $y\in C^1(0,\infty)$. Differentiating once more
at the points of continuity of $f$ we obtain
$$
y''=\psi_{i,\downarrow}''A+\psi_{i,\uparrow}''B
+\bigl[\psi_{i,\downarrow}'\psi_{i,\uparrow}
-\psi_{i,\uparrow}'\psi_{i,\downarrow}\bigr]f\,m_i ,
$$
where the bracket equals $\mathcal W_i=-|\mathcal W_i|$, so that the last term
equals $-2f(u)/(\sigma_i^2u^2)$. Substituting this into
$\mathcal L^0_iy-\tilde\lambda_iy$ and using that $\psi_{i,\uparrow}$ and
$\psi_{i,\downarrow}$ solve the homogeneous equation, we obtain
\eqref{eq:resolvent_ode}; the continuity of $y''$ follows from that of $f$.
\qed
\end{proof}

\begin{theorem}
\label{thm:annuity_case}
Let $\sigma_i>0$, $c_i<0$ and let $F_i$ be an arbitrary probability measure on
$(0,\infty)$ for every regime $i\in E$. Then the survival probability
$\Phi_i(u)$ belongs to $C^2(0,\infty)$.
\end{theorem}

\begin{proof}
By Lemma~\ref{lem:repr}, $\Phi_i=\mathcal K_i[\mathcal H_i(\Phi)]$. Since
$\Phi_i$ is bounded, being a probability, and the $F_i$ are probability
measures, the function $\mathcal H_i(\Phi)$ is bounded as well, and
Lemma~\ref{lem:LtoC} gives $\Phi_i\in C^1(0,\infty)$; in particular all
$\Phi_j$ are continuous.

Since the support of $F_i$ is contained in $(0,\infty)$, the integral term is
continuous by Lemma~\ref{lem:convolution}{\rm(i)}, and so is
$\mathcal H_i(\Phi)$. The second part of Lemma~\ref{lem:LtoC} yields
$\Phi_i\in C^2(0,\infty)$.
\qed
\end{proof}

\section{The mixed case}
\label{s6}

In this section the sign of the drift is arbitrary and
Lemma~\ref{lem:asymptotics} covers all three cases. For $c_i\ge0$ zero is
inaccessible, so $\zeta=\infty$, and for $\psi_{i,\uparrow}$ one takes the
solution bounded at zero. Let us check that the kernel \eqref{eq:green} is
still integrable at zero. For $c_i>0$ the numerator $\psi_{i,\uparrow}(v)$
tends to $\psi_{i,\uparrow}(0+)>0$ while $|\mathcal W_i(v)|$ grows like
$e^{2c_i/(\sigma_i^2v)}$, so that the ratio
$\psi_{i,\uparrow}(v)/\bigl(v^2|\mathcal W_i(v)|\bigr)$ tends to zero. For
$c_i=0$ this ratio equals $v^{\rho_+-2+2a_i/\sigma_i^2}/|C_{\mathcal W}|$ and
is integrable at zero, since $\rho_+>\rho^*$. The behaviour at infinity does
not depend on the sign of $c_i$. Therefore Lemmas~\ref{lem:repr} and
\ref{lem:LtoC} remain valid and, in particular, $\Phi_i\in C^1(0,\infty)$ for
all $c_i$ and $F_i$. In contrast with the annuity case, going further requires
additional conditions.

\begin{theorem}
\label{thm:mixed_case}
Assume that in the mixed model $\sigma_i>0$ and that for every regime $i\in E$
at least one of the following conditions holds:
{\rm(A)} the measure $F_i$ has no atoms on $(-\infty,0)$;
{\rm(B)} $\Phi_i(0+)=0$.
Then the survival probability $\Phi_i(u)$ belongs to $C^2(0,\infty)$.
Condition {\rm(B)} holds, in particular, for $c_i<0$.
\end{theorem}

\begin{proof}
As noted above, $\Phi_i\in C^1(0,\infty)$, so all $\Phi_j$ are continuous and
bounded. If {\rm(A)} holds, then $F_i(\{-u\})=0$ for all $u>0$, that is,
condition {\rm(ii)} of Lemma~\ref{lem:convolution} is satisfied; if {\rm(B)}
holds, then condition {\rm(iii)} is satisfied. In both cases the integral term,
and with it $\mathcal H_i(\Phi)$, is continuous, and the second part of
Lemma~\ref{lem:LtoC} gives $\Phi_i\in C^2(0,\infty)$.

Let $c_i<0$. Survival requires an event before the diffusion reaches zero,
hence
$$
\Phi_i(u)\le\bbp(T<\zeta)=1-\bbe\bigl[e^{-\tilde\lambda_i\zeta}\bigr]
=1-\frac{\psi_{i,\downarrow}(u)}{\psi_{i,\downarrow}(0+)}
\;\xrightarrow[u\to0+]{}\;0 .
$$
Thus {\rm(B)} holds.
\qed
\end{proof}

\begin{proposition}[Failure of smoothness at an atom]
\label{prop:counterexample}
Let $K=1$, $\sigma>0$, $\alpha>0$, let the measure $F$ have an atom at a point
$-b<0$ of weight $\varkappa:=F(\{-b\})>0$ and no other atoms in some punctured
neighbourhood of $-b$, and let $\Phi(0+)>0$. Then $\Phi\in C^1(0,\infty)$, the
second derivative has one-sided limits at $b$, and
\begin{equation}
    \label{eq:jump}
    \Phi''(b^+)-\Phi''(b^-)=-\frac{2\alpha\,\varkappa\,\Phi(0+)}{\sigma^2b^2}
    \;<\;0 ,
\end{equation}
in particular $\Phi\notin C^2(0,\infty)$. The condition $\Phi(0+)>0$ holds,
for instance, for $c>0$.
\end{proposition}

\begin{proof}
Membership in $C^1$ was noted at the beginning of the section. The term of
$\mathcal H(\Phi)$ corresponding to the atom equals
$\alpha\varkappa\,\Phi(u-b)\Ind_{\{u>b\}}$. The remaining part corresponds to
the measure $\widetilde F:=F-\varkappa\delta_{-b}$, for which
$\widetilde F(\{-u\})=0$ for all $u$ in some neighbourhood of $b$, so that by
Lemma~\ref{lem:convolution}{\rm(ii)} it is continuous there. Hence
$\mathcal H(\Phi)$ is continuous in a punctured neighbourhood of $b$, has
one-sided limits there, and
$$
\mathcal H(\Phi)(b^+)-\mathcal H(\Phi)(b^-)=\alpha\varkappa\,\Phi(0+) .
$$
The computation in the proof of Lemma~\ref{lem:LtoC} shows that the first two
terms in the expression for $y''$ are always continuous, while the last one has
at $b$ the jump $-2\bigl(f(b^+)-f(b^-)\bigr)/(\sigma^2b^2)$; substituting
$f=\mathcal H(\Phi)$ gives \eqref{eq:jump}.

Finally, let $c>0$. Put $A:=\{|W_t|\le1\ \text{for all}\ t\le1\}$, so that
$q:=\bbp(A)>0$; the function $\mathcal E$ is defined in \eqref{eq:Yexpl}. On
the event $A$ one has
$\mathcal E_t\mathcal E_s^{-1}\ge e^{-|a-\sigma^2/2|-2\sigma}$ for all
$0\le s\le t\le1$, whence
$Y^u_1\ge L:=c\,e^{-|a-\sigma^2/2|-2\sigma}>0$ for all $u>0$, and moreover
$Y^u_t>0$ for $t\le1$. Intersecting $A$ with the event ``no event occurred
before time $1$'', which is independent of $W$ and has probability
$e^{-\alpha}>0$, and using the monotonicity of $\Phi$, we get
$\Phi(u)\ge q\,e^{-\alpha}\,\Phi(L)$ for all $u>0$. It remains to check that
$\Phi(L)>0$. By the non-triviality assumption and monotonicity there is $u_1$
with $\Phi(u_1)>0$. Since $Y^{L}_1\ge L\,\mathcal E_1$ and $\mathcal E_1$ has a
lognormal distribution with support $(0,\infty)$, the probability
$r:=\bbp(L\mathcal E_1\ge u_1)$ is positive; intersecting this event with the
absence of events before time $1$ we obtain
$\Phi(L)\ge r\,e^{-\alpha}\,\Phi(u_1)>0$.
\qed
\end{proof}

\begin{remark}
The jump \eqref{eq:jump} is non-zero if and only if $F$ has an atom on the
negative half-line and, simultaneously, $\Phi(0+)>0$. Thus the boundary of
smoothness is determined by the value $\Phi(0+)$ and not by the sign of the
drift: for $c<0$ one always has $\Phi(0+)=0$ and smoothness is preserved, for
$c>0$ one always has $\Phi(0+)>0$, and for $c=0$ both alternatives are
possible. The last statement follows from Lemma~\ref{lem:repr}. For $c=0$ zero
is inaccessible and $Y_t=u\mathcal E_t\to0$ almost surely as $u\to0+$. Since
the measure $F$ does not charge zero and $\Phi$ is continuous, the dominated
convergence theorem gives
$\mathcal H(\Phi)(y)\to\alpha\int_{(0,\infty)}\Phi\,\diff F$ as $y\to0+$,
whence for $K=1$
$$
\Phi(0+)=\int_{(0,\infty)}\Phi(z)\,F(\diff z) .
$$
Since $\Phi>0$ on $(0,\infty)$, the argument in the proof of
Proposition~\ref{prop:counterexample} being applicable for every $c\ge0$
because $Y_t\ge u\mathcal E_t$, the right-hand side is positive exactly when
$F((0,\infty))>0$. An admissible measure in the mixed model is, for example,
$$
F=\tfrac12\delta_{-b}+\tfrac12\delta_{+b} .
$$
It has finite moments of all orders.
\end{remark}

By Theorems~\ref{thm:annuity_case} and \ref{thm:mixed_case} and equality
\eqref{eq:resolvent_ode} with $f=\mathcal H_i(\Phi)$, for all regimes $i\in E$
and $u>0$ the vector of survival probabilities satisfies the system of
integro-differential equations, cf.\ \cite[Proposition~2]{AntipovKabanov},
\begin{multline*}
    \tfrac12\sigma_i^2u^2\Phi_i''(u)+(a_iu+c_i)\Phi_i'(u)
    +\sum_{j\in E}\lambda_{ij}\Phi_j(u)+ \\
    +\alpha_i\int_{-\infty}^{\infty}
    \bigl(\Phi_i(u+z)\Ind_{\{u+z>0\}}-\Phi_i(u)\bigr)F_i(\diff z)=0,
    \quad i\in E.
\end{multline*}
Here the convention $\lambda_{ii}=-\sum_{j\ne i}\lambda_{ij}$ is used, so that
the term $-\alpha_i\Phi_i$ enters exactly once, inside the integral.

\end{document}